\documentclass [11pt,reqno]{amsart}

\usepackage{graphicx}
\usepackage{indentfirst}
\usepackage[mathscr]{eucal}
\usepackage[all]{xy}
\usepackage{amsmath,amsxtra,amssymb,latexsym,amscd,amsthm,wasysym, amsfonts,color,esint,mathtools,geometry}
\usepackage{fancyhdr}
\usepackage{dsfont}
\usepackage{comment}
\usepackage{enumitem}

\def\PSH{\mathcal{PSH}}
\def\SH{\mathcal{SH}}

\def\ddc{dd^c}

\def\Capa{\mathrm{Cap}}

\def\Vol{\mathrm{Vol}}

\numberwithin{equation}{section}
\def\1{\mathds{1}}

\usepackage{cite}
\usepackage[backref,pagebackref,pdftex,hyperindex]{hyperref}
\hypersetup{
    unicode=false,        
    pdftoolbar=true,      
    pdfmenubar=true,       
    pdffitwindow=false,     
    pdfstartview={FitH},    
    pdftitle={},    
    pdfauthor={},     
    colorlinks=true,       
   linkcolor=blue,          
    citecolor=blue,        
    filecolor=black,      
    urlcolor=black} 

\title[Integrability of $(\omega,m)$-subharmonic functions ]{Integrability of $(\omega,m)$-subharmonic functions on compact Hermitian manifolds}

\author{Yuetong Fang}

\address{Université d'Angers, CNRS, LAREMA, SFR MATHSTIC, F-49000 Angers, France}
\email{yuetong.fang@univ-angers.fr}

\begin{document}

\begin{abstract}
    Let $(X,\omega)$ be a compact Hermitian manifold of dimension $n$. We show that all $(\omega,m)$-subharmonic functions are $L^p$ integrable on $X$, for any $p < \frac{n}{n-m}$. 
\end{abstract}
\maketitle

%mathbb letters
\newcommand{\A}{\mathbb{A}}
\newcommand{\B}{\mathbb{B}}
\newcommand{\C}{\mathbb{C}}
\newcommand{\G}{\mathbb{G}}

\newcommand{\Gm}{\mathbb{G}_\mathrm{m}}

\newcommand{\N}{\mathbb{N}}
\renewcommand{\P}{\mathbb{P}}
 \newcommand{\Q}{\mathbb{Q}}
 \newcommand{\R}{\mathbb{R}}
 \newcommand{\Z}{\mathbb{Z}}
\newcommand{\T}{\mathbb{T}}

%accent
\newcommand{\poll}{ł}
\newcommand{\swea}{å}

\newtheorem{theorem}{Theorem}[section]
\newtheorem{thm}[theorem]{Theorem}
\newtheorem{lemma}[theorem]{Lemma}
\newtheorem{lem}[theorem]{Lemma}
\newtheorem{prop}[theorem]{Proposition}
\newtheorem{coro}[theorem]{Corollary}
\newtheorem{corollary}[theorem]{Corollary}
\newtheorem*{Main Theorem}{Main Theorem}
\newtheorem*{Theorem B}{Theorem B}

\theoremstyle{definition}
\newtheorem{definition}[theorem]{Definition}
\newtheorem{defi}[theorem]{Definition}
\newtheorem{example}[theorem]{Example}
\newtheorem{exa}[theorem]{Example}
\newtheorem{claim}[theorem]{Claim}
\newtheorem{remark}[theorem]{Remark}
 \newtheorem*{ackn}{Acknowledgements}

\section{Introduction}
Let $\Omega \subset \C^n$ be an open set and $\omega$ a hermitian $(1,1)$-form on $\Omega$. Let $u$ be a real $\mathcal{C}^2$ functions on $\Omega$ such that the eigenvalues $\lambda = (\lambda_1, \cdots, \lambda_n)$ of the complex Hessian matrix $[u_{i \bar{j}}]_{1\le i,j \le n}$, belong to the closure of the cone
\[
\Gamma_{m} = \{ \lambda\in \R ^n : S_1(\lambda)>0, \cdots, S_m(\lambda)> 0\},
\]
where $S_k(\lambda)$ denotes the $k$-th elementary symmetric function of $\lambda$:
\[
S_k (\lambda) = \sum_{0 < j_1 < \cdots < j_k \le n} \lambda_{j_1} \lambda_{j_2} \cdots \lambda_{j_k}.
\]
Such a function is called $m$-subharmonic ($m$-sh). As shown by B\poll ocki in \cite{Blocki05}, $m$-sh functions are the right class of admissible solutions to the complex Hessian equations
\[
(d d^c u )^m \wedge \omega^{n-m} = f \omega^n,
\]
which are the generalization of Poisson equations ($m=1$) and Monge--Amp\`ere equations ($m=n$).

While all plurisubharmonic functions are locally $L^p$-integrable for all $p \ge 1$, this is not necessarily the case for $m$-sh functions. The fundamental solution to the $m$-Hessian equation when $\omega=\ddc|z|^2$ is given by $-\frac{1}{|z|^{2n/m-2}}$, which is $L^p$-integrable only when $p< \frac{nm}{n-m}$. B\l ocki conjectured that all $m$-sh functions are locally $L^p$-integrable for $p< \frac{nm}{n-m} $. This conjecture is partially confirmed in \cite{DK14} \cite{AC20}, where the former developed a useful tool to study the integrability of $m$-sh functions, known as volume-capacity inequality:
\[
\Vol (K) \le C_{p} \Capa^{p}_m(K,\Omega),
\]
with $K$ being a compact subset of $\Omega$, $\Capa_m(K,\Omega)$ denote the Hessian capacity, and $p < \frac{n}{n-m}$. 

On compact hermitian manifolds $(X,\omega)$, we consider the $m$-Hessian operator
\[
H_m(u) = (\omega+d d^c u)^m \wedge \omega^{n-m}.
\]
A $\mathcal{C}^2$ function is called $(\omega,m)$-subharmonic with respect to $\omega$ if $H_k(u) \ge 0$ for all $k=1, \cdots,m$. The study of $m$-Hessian equations on compact manifolds (see for example \cite{Sze18}, \cite{Zhang_2017}, \cite{Ko_odziej_2016}, \cite{guedj2023degeneratecomplexhessianequations}, \cite{kolodziej2023complexhessianmeasuresrespect}) motivate the development of potential theory for $(\omega,m)$-sh functions on hermitian manifolds. The main purpose of this article is to extend Dinew--Ko\l odziej's integrability result to the Hermitian manifolds. We now present the precise statement.
\begin{Main Theorem} \label{Main Theorem}
    Let $(X,\omega)$ be a compact complex hermitian manifold of dimension $n$, equipped with a hermitian form $\omega$. Fix an integer $1\le m \le n$.  For any $p < \frac{n}{n-m}$, we have $ \SH_{m}(X,\omega) \subset L^p(X, \omega^n)$.
\end{Main Theorem}
When $(X,\omega)$ is a K\"ahler manifold, the theorem is proved in \cite{LN15}, using the volume-capacity inequality provided by the work of Dinew and Ko{\l}odziej \cite{DK14}, and the Chern--Levine--Nirenberg(CLN) inequality. The proof of the latter inequality goes through several integration by parts which are delicate in the hermitian setting due to the appearance of torsion terms. In our setting when $m<n$, there is also a lack of positivity since $(\omega,m)$-subharmonic functions are in general not $\omega$-plurisubharmonic. In this short note, we observe that in the volume capacity estimate of Dinew and Ko{\l}odziej the candidates defining the capacity can be taken to be $\omega$-plurisubharmonic. This observation simplifies several potential estimates that we carry out in Section~\ref{Sec:int}, where we prove the main result. In  Section~\ref{Sec:Preliminaries}, we focus on reviewing essential concepts and results required for the proof. 

\medskip

\noindent {\bf Acknowledgment.} The author thanks her supervisor, Hoang-Chinh Lu, for suggesting the problem, valuable discussions, and useful comments that improved the presentation. This research is part of a PhD program funded by the PhD scholarship CSC-202308070110. This work is partially supported by the projects Centre Henri Lebesgue ANR-11-LABX-0020-01 and PARAPLUI ANR-20-CE40-0019.

\section{Preliminaries} \label{Sec:Preliminaries}
Throughout this paper, we denote $(X,\omega)$ as a compact Hermitian manifold of complex dimension $n\in \mathbb{N}^{*}$, equipped with a Hermitian metric $\omega$. We use the differential operators $d = \partial + \bar{\partial}$, and $d^c = i (\bar{\partial}- \partial)$, so that $\ddc = 2i \partial \bar{\partial}$. 

We now recall the notions of $(\omega,m)$-subharmonic functions and gather some results related to $\omega$-plurisubharmonic functions and $(\omega,m)$-subharmonic functions. 

\subsection{\texorpdfstring{$(\omega,m)$}{}-subharmonic functions}
Fix an integer $1 \le m \le n$. Fix $\Omega $ an open set in $\C^n$. We follow the definition in \cite{GuNguyen2018}.

\begin{definition}
A real $(1,1)$-form $\alpha$ on $X$ is called $m$-positive with respect to $\omega$ if at all points in $X$,
\[
\alpha^{k} \wedge \omega^{n-k} \ge 0, \forall k =1, \cdots, m.
\]
\end{definition}

\begin{definition}
Given a hermitian metric $\alpha$ on $\Omega$, a $\mathcal{C}^2(\Omega)$ function $u : \Omega \rightarrow  \R$ is called 
 harmonic with respect to $\alpha$ if $\ddc u \wedge \alpha^{n-1}=0$ at all points in $\Omega$.
\end{definition}

\begin{definition}
A function $u:\Omega \rightarrow \{ -\infty\} \cup \R$ is subharmonic with respect to $\omega$ if:
\begin{enumerate}[label=(\alph*)]
\item $u$ is upper semicontinuous and $u \in L^1_{loc}(\Omega)$;
\item for every relatively compact open $D \Subset \Omega$ and every function $h\in \mathcal{C}^0(D)$ that is harmonic with respect to $\omega$ on $D$ the following implication holds
\[
u \le h \text{ on }\partial D \implies u \le h \text{ in }D.
\]
\end{enumerate}
\end{definition}

\begin{definition}
A function $\varphi : \Omega \rightarrow  \{ -\infty\} \cup \R$ is quasi-subharmonic with respect to $\alpha$ if locally $\varphi = u + \rho$, where $u$ is subharmonic with respect to $\alpha $ and $\rho$ is smooth.

A function $\varphi$ is $\omega$-subharmonic with respect to $\alpha$ if $\varphi$ is quasi-subharmonic with respect to $\alpha$ and $(\omega + \ddc \varphi) \wedge \alpha^{n-1}\ge 0$ in sense of distributions.
\end{definition}
The positive cone $\Gamma_m(\alpha)$ associated with the metric $\alpha$ is defined by 
\[
\left\{  \gamma \text{ real }(1,1) \text{-form}: \gamma^k\wedge \alpha^{n-k}>0, k=1,\cdots,m \right\}.
\]
It follows from Gårding's inequality\cite{gar59} that if $\gamma_0, \gamma_1, \cdots, \gamma_{m-1}\in \Gamma_{m}(\alpha)$, then 
\[
\gamma_0 \wedge\gamma_1\wedge \cdots \wedge\gamma_{m-1} \wedge\alpha^{n-m} >0.
\]
One can write $\widetilde{\alpha}= \gamma_1\wedge \cdots \wedge\gamma_{m-1} \wedge\alpha$, and it is a strictly positive $(n-1,n-1)$-form on $\Omega$.
\begin{definition}
A function $\varphi: \Omega \rightarrow [ -\infty, + \infty)$ is called $(\omega,m)$-subharmonic with respect to $\alpha$ if $\varphi$ is $\omega$-subharmonic with respect to $\widetilde{\alpha}$ in $\Omega$ for all 
$\widetilde{\alpha}$ of the form $\widetilde{\alpha}^{n-1}= \gamma_1\wedge \cdots \wedge\gamma_{m-1} \wedge\alpha$, where $\gamma_1, \cdots,\gamma_{m-1} \in \Gamma_m(\alpha)$.

A function $u : X \rightarrow \mathbb{R} \cup \{- \infty \} $ is called $(\omega,m)$-subharmonic on $X$ if $u$ is $(\omega,m)$-subharmonic on each local chart $U$ of $X$.
\end{definition}
The set of all locally integrable functions on $U$ which are $(\omega,m)$-subharmonic with respect to $\alpha$ in $U$ is denoted by $\SH_{\alpha,m}(U,\omega)$. We remark that $\omega$ and $\alpha$ are not necessarily the same.

However, in this paper, we focus on the set $\SH_{\omega,m}(U,\omega)$. To simplify the notations, we denote $\SH_m(U,\omega)$ the set of all $(\omega,m)$-subharmonic functions with respect to $\omega$ in $U$. The set of all $\omega$-plurisubharmonic functions on $U$ is denoted by $\PSH(U, \omega)=\SH_n(U,\omega)$.

We denote $\omega_{u} \coloneqq \omega + \ddc u$.

\begin{remark}
  By Gårding's inequality\cite{gar59}, if $u \in \mathcal{C}^2(X)$, then $u$ is $(\omega,m)$-subharmonic with respect to $\omega$ on $X$ if and only if the associated form $\omega _u$ belongs to the closure of $\Gamma_m(\omega) $.
\end{remark}

The integration by parts formula is valid for $\mathcal{C}^2 (X)$ functions (by Stokes' theorem). We can see that it still holds for bounded $(\omega,m)$-subharmonic functions by \cite[Proposition 3.20]{kolodziej2023complexhessianmeasuresrespect}. We will need the following version. 
\begin{prop}[Integration by parts] \label{For:IPP}
    Let $\varphi, \psi \in \SH_m(X,\omega) \cap C^2(X)$. Let $T$ be a smooth $(n-1,n-1)$-form. Then,
    \[
    \int_X \varphi \ddc \psi \wedge T =\int_X \psi \ddc \varphi \wedge T + 2 \int_X \psi d \varphi \wedge d^cT + \int_X \psi \varphi  \ddc T.
    \]
\end{prop}
\begin{proof}
Observe that $\int_X\varphi\partial \psi \wedge \partial T= \int_X \varphi \bar \partial \psi \wedge \bar{\partial} T= 0 $ by bidegree reason. Hence, \[
\int_X\varphi d^c \psi \wedge dT= -\int_X\varphi d\psi \wedge d^cT.
\]
It thus follows from Stokes' theorem that
\[
    \begin{aligned}
    \int_X \varphi \ddc \psi \wedge T =& \int_X \psi \ddc ( \varphi  T)\\
    =& \int_X \psi d(d^c \varphi \wedge T + \varphi d^cT)\\
    =& \int_X \psi \left(dd^c \varphi \wedge T - d^c \varphi \wedge dT+ d\varphi \wedge d^cT +\varphi \ddc T\right)\\
    =&\int_X \psi \ddc \varphi \wedge T + 2 \int_X \psi d \varphi \wedge d^c T + \int_X \psi \varphi  \ddc T.
    \end{aligned}
    \]
\end{proof}

 We state an important result for $(\omega,m)$-subharmonic functions (see \cite{Ko_odziej_2016}, Lemma 2.3).
\begin{lem} \label{lem:inepc}
    Let $u \in \SH_m(\Omega,\omega)\cap \mathcal{C}^2(\Omega)$ and $T$ be a smooth $(n-k,n-k)$-form with $1\le k \le m-1$. Then,
    \[
    |\omega_u^k \wedge T /\omega^n| \le C_{n,k,\|T\|} \omega_u^k \wedge \omega^{n-k} /\omega^n,
    \]
    where $C_{n,k,\|T\|}$ is a uniform constant depending only on $n,k$ and $\sup_{\Omega}\|T\|$.
\end{lem}
We remark that the assumption $k<m$ is crucial.  

We end this subsection with a $L^1$-compactness result (see \cite[lemma 3.3]{Ko_odziej_2016}).

\begin{lem} \label{lem:cap}
    Let $u\in \SH_m(X,\omega)$ be normalized by $\sup_X u =0$. Then there exists a uniform constant $A>0 $ depending only on $X, \omega$ such that
    \[
    \int_X |u|\omega^n \le A.
    \]
\end{lem}

%In particular, when $n=m$, the set $\PSH(X,\omega)$ is a closed subset of $L^1(X)$, for the $L^1$ topology. 

\subsection{The Cauchy--Schwarz inequality}
Let $h$ be a smooth real-valued function and $u,v$ be Borel functions. Let $T$ be a positive current of bidegree $(n-2,n-2)$. We state the following Cauchy\textendash Schwarz inequality (see \cite[Proposition 1.4]{NG16}, \cite[Lemma 2.3]{kolodziej2023complexhessianmeasuresrespect}).
\begin{lem} \label{lem:cs}
    There exists a uniform constant $C$ depending on $\omega$ such that
\[
\left|  \int_X uv dh \wedge d^c \omega \wedge T \right|^2 \le C \int_X |u|^2 d h \wedge d^c h \wedge \omega \wedge T \int_X |v|^2  \omega^2 \wedge T.
\]    
\end{lem}
Although the Cauchy–Schwarz inequality typically requires $T$ to be positive, it still holds for $T = \gamma^{m-1}\wedge \omega^{n-m-1}$, which is not necessarily positive, with $\gamma$ being $m$-positive with respect to $\omega$(see \cite[lemma 2.4]{kolodziej2023complexhessianmeasuresrespect}).
\subsection{Capacity}
Let $E \subset X$ be a Borel subset.
\begin{definition}
     The $(\omega, m)$-capacity of $E$ is defined by 
    \[
    \Capa_{\omega,m}(E) \coloneqq \sup \left\{ \int_E \omega_{\varphi}^m \wedge \omega^{n-m} : \varphi \in \SH_m(X, \omega), 0\le \varphi \le 1 \right\}.
    \]
\end{definition}
It's clear that the $(\omega, m)$-capacity is well defined for $\varphi \in \mathcal{C}^2(X)$. Thanks to \cite[theorem 3.3]{kolodziej2023complexhessianmeasuresrespect}, the Hessian operator for bounded $(\omega,m)$-subharmonic function is well defined, and so is the capacity. 

Since we will need an upper bound for the capacity of sublevel sets $\{ \psi <-t\}$, controlling the terms $d \varphi \wedge d^c \omega^{p} $ when $\varphi \in \SH_m(X,\omega)$ (which arises from integration by parts) can be difficult. We therefore introduce the following version of the Hessian capacity.
\begin{definition}
  We define: 
    \[
    \widetilde{\Capa}_{\omega,m}(E) \coloneqq \sup \left\{ \int_E \omega_{\varphi}^m \wedge \omega^{n-m} : \varphi \in \PSH(X, \omega), 0\le \varphi \le 1 \right\}.
    \]
\end{definition}
 We observe that $ \widetilde{\Capa}_{\omega,m}(E) \le \Capa_{\omega,m}(E)$, since $\PSH(X,\omega)= \SH_{n}(X,\omega) \subset \SH_{m}(X,\omega)$.

\section{Integrability of \texorpdfstring{$(\omega,m)$-}{}subharmonic functions on Hermitian manifolds} \label{Sec:int}
 \begin{prop}[Volume-Capacity estimate]  \label{prop:V-Cest}
For $ 1<\tau < \frac{n}{n-m}$, there exists a constant $C_{\tau}$ such that for each Borel subset $K$ of $X$,
\[
V(K) \le C_{\tau} \widetilde{\Capa}_{\omega,m}^{\tau}(K),
\]
where $V(K)= \int_K \omega^n$.
 \end{prop}
\begin{proof}
The argument presented in \cite[Proposition 3.6]{Ko_odziej_2016} is still valid here, as we use the estimates for $\omega$-plurisubharmonic solutions to the Monge--Amp\`ere equations. For the sake of completeness, we provide a slightly different proof below. 

We normalize the volume form such that $\int_X \omega^n = 1$. We can assume that $V(K)>0$, otherwise the inequality to be proved is trivial. Fixing $p>1$, we solve the complex Monge--Amp\`ere equation 
\begin{equation*} \tag{MA} \label{equation:MA}
(\omega+dd^c u)^n = a f_{K} \omega^n, \; \sup_X u=0, \text{ where }f_K=\left ( V(K)^{-1/p} {\bf 1}_K + 1 \right).
\end{equation*}
The existence of $a>0$ and $u\in \PSH(X,\omega)\cap L^{\infty}$ follows from \cite{KN15MA}. We next show how to control the constant $a>0$ uniformly. Observe first that $(\omega+dd^c u)^n \geq a \omega^n$, hence the domination principle,  \cite[Corollary 1.13]{GL21}, gives $a\leq 1$. 
We also have
\[
1=\|V(K)^{-1/p} {\bf 1}_K \|_p\le \|f_K \|_p \le \|V(K)^{-1/p} {\bf 1}_K \|_p + \|1\|_p \le 2.
\]
It thus follows from \cite[Theorem 2.1]{GL21} that 
\[
u \ge -C_1,
\]
where $C_1\geq 1$ is a uniform constant, independent of $K$.
We also deduce from Step 1 of \cite[Theorem 2.1]{GL21} that there exists a uniform constant $C_2>0$, independent on $K$, and a bounded $\omega$-psh function $\psi$, $-1\leq \psi\leq 0$, such that 
\[
(\omega + \ddc \psi)^n \ge C_2^{-1} f_K \omega^n.
\]
We conclude from the domination principle, see \cite[Corollary 1.13]{GL21}, that $C_2^{-1} \le a$.

On the other hand, the mixed Monge--Amp\`ere inequality \cite{NG16} yields that 
\[
\omega_u ^m \wedge \omega^{n-m} \ge a^{m/n}f_K ^{m/n} \omega^n.
\]
For $v=C_1^{-1}u$, we then have
\[
\widetilde{\Capa}_{\omega,m}(K) \ge   \int_K (\omega + \ddc v)^m \wedge \omega^{n-m} \ge a^{m/n} C_1^{-m}  V(K)^{1- \frac{m}{pn}} \ge {C_2}^{-m/n} C_1^{-m}  V(K)^{1- \frac{m}{pn}}.
\]
Hence, for all $p>1$, there exist $C_p$ such that 
\[
V(K) \le C_p \widetilde{\Capa}_{\omega,m}(K)^{\frac{pn}{pn-m}},
\]
which completes the proof.
%We fix an open set $U$ such that $K \subset U$. We look for $u \in \PSH(X,\omega)$, with $\sup_X u =0$, solving the complex Momge\textendash Amp\`ere equation 
%\[
%\omega_u ^n = b f \omega^n, \text{where } f= V(U)^{-1} \chi_U.
%\]
%The existence of continuous solutions $u \in \PSH(X,\omega)$ and constant $b>0$ follows from [KN19]. Moreover,  
\end{proof}
 
 \begin{lem} \label{3.1}
Let $\varphi$ be a bounded $\omega$-plurisubharmonic function satisfying $-1\le \varphi \le 0$, and $\psi$ be a bounded $(\omega,m)$-subharmonic function with respect to $\omega$, normalized by $\sup_X \psi=0$. Then, there exists a uniform constant $C >0$, depending on $\omega$, $n$, $m$, such that the following inequality holds:
\[
 \int_{X} \left| \psi  \right| \omega_{\varphi}^m \wedge \omega^{n-m} \le C.
\]
 \end{lem}
We emphasize here that the constant $C$ does not depend on the $L^{\infty}$-norm of $\psi$.
\begin{proof}
By approximation, we can also assume that $\varphi$ and $\psi$ are smooth. We will prove the following bound
\[
    \int_{X} \left| \psi  \right| \omega_{\varphi}^k \wedge \omega^{n-k} \le C_k
\]
for all $k \in \{ 1, \cdots, m\}$, by induction on $k$. Here $C_k$ are uniform constants independent of $\psi$ and $\varphi$.
    We denote by $K_1, K_2, \cdots$ uniform positive constants. We fix $A$ such that
\[
\int_X |u| \omega^n \le A, \text{ for all } u \in \SH_m(X,\omega)\cap L^{\infty}(X), \; \sup_X u=0.
\]    
We remark that there exists $B \ge 0$, such that 
\[
    -B \omega^{2} \le \ddc \omega \le B \omega^2 \text{ and } -B \omega^{3} \le d\omega \wedge d^c \omega \le B \omega^{3}.
\]
Up to enlarging $B$, we can also assume that 
\begin{equation} \label{for:B} \tag{B}
-B \omega^{k+1} \le \ddc (\omega^{k}) \le B \omega^{k+1}, \; \forall k.
\end{equation}
 First we prove the bound for $k=1$. Using integration by parts (Proposition~\ref{For:IPP}), and the compactness result (Lemma~\ref{lem:cap}), we obtain
    \[
\begin{aligned}
\int_X|\psi| \omega_{\varphi} \wedge \omega^{n-1}= & \int_X(-\psi) \omega^n+\int_X(-\psi) d d^c \varphi \wedge \omega^{n-1} \\
\le & A+\int_X(-\varphi) d d^c \left(\psi \wedge \omega^{n-1}\right) \\
= & A +\int_X(-\varphi) d d^c \psi \wedge \omega^{n-1} \\
& +2 \int_X(-\varphi) d \psi \wedge d^c \omega^{n-1}+\int_X(-\varphi) \psi  d d^c \omega^{n-1}.
\end{aligned}
    \]
A basic computation yields:
\[
\int_X(-\varphi) d d^c \psi \wedge \omega^{n-1} = \int_X(-\varphi) \omega_{\psi} \wedge \omega^{n-1} + \int_X \varphi \omega^{n}.
\]
The assumption $-1 \le \varphi \le 0$ implies that $\int_X \varphi \omega^{n} \le 0$. Moreover, due to the positivity of $\omega_{\psi} \wedge \omega^{n-1}$, it follows that \[
\int_X(-\varphi) \omega_{\psi} \wedge \omega^{n-1} \le\int_X (\sup_X|\varphi|) \omega_{\psi} \wedge \omega^{n-1} \le\int_X  \omega_{\psi} \wedge \omega^{n-1}.
\]
Moreover,
\[
\begin{aligned}
\int_X  \omega_{\psi} \wedge \omega^{n-1} =& \int_X \omega^n + \int_X \ddc \psi \wedge \omega^{n-1}\\
=&  \int_X \omega^n + \int_X \psi \ddc  (\omega^{n-1}). 
\end{aligned}
\]
Together with \eqref{for:B} and the compactness result, we thus obtain that
\[
\int_X(-\varphi) d d^c \psi \wedge \omega^{n-1} \le V + B A.
\]
Furthermore, by the same computation as above,
\begin{equation} \label{Formula:comp+B}
\begin{aligned}
 \int_X(-\varphi) \psi \wedge d d^c \omega^{n-1}
 \le & B  \int_X (-\varphi) (-\psi) \omega^n
 \le B A.
 \end{aligned}
\end{equation}
We need to obtain an upper bound for $\int_X(-\varphi) d \psi \wedge d^c  \omega^{n-1}$. After an elementary computation, Stokes' theorem and \eqref{Formula:comp+B} yield
\[
\begin{aligned}
\int_X(-\varphi) d \psi \wedge d^c \omega^{n-1} 
&= \int_X \psi d \varphi \wedge d^c \omega^{n-1} + \int_X \psi \varphi d d^c \omega^{n-1} \\
& \le \int_X \psi d \varphi \wedge d^c \omega^{n-1} + BA.
\end{aligned}
\]
The Cauchy-Schwarz inequality (Lemma~\ref{lem:cs}) yields
\[
\begin{aligned}
-\int_X(- \psi) d \varphi \wedge d^c \omega^{n-1} &=- (n-1) \int_X (-\psi) d \varphi \wedge d^c \omega \wedge  \omega^{n-2}
\\
&\le K_1 \left( \int_X (-\psi) d \varphi \wedge d^c \varphi \wedge \omega^{n-1} \right)^{1/2} \left( \int_X (-\psi)\omega^n \right)^{1/2} \\
&\le K_1 A^{1/2} \left( \int_X (-\psi) d \varphi \wedge d^c \varphi \wedge \omega^{n-1} \right)^{1/2}.    
\end{aligned}
\]
We claim that
\begin{equation}  \label{dweddc}
  d \varphi \wedge d^c \varphi \leq \omega + \frac{1}{2} d d^c (\varphi +1)^2.  
\end{equation}
Indeed, $\omega + 1/2 d d^c (\varphi +1)^2 - d\varphi \wedge d^c \varphi  = \omega + (\varphi+1) d d^c \varphi = (\varphi+1) \omega_{\varphi} - \varphi \omega $ is positive because of the $\omega$-plurisubharmonicity of $\varphi$ and the assumption $-1\leq \varphi \leq 0$.

So we obtain that
\[
\int_X(-\varphi) d \psi \wedge d^c \omega^{n-1} \le BA + K_1A \left( \int_X (-\psi) (\omega + 1/2 d d^c (\varphi +1)^2 ) \wedge \omega^{n-1}\right)^{1/2}.
\]
To continue we denote 
$$
A(\psi) = \sup \left\{ \int_X (-\psi) \omega_u \wedge \omega^{n-1} : 0 \leq u \leq 1, u \in \PSH(X,\omega)\cap C^2(X) \right\}.
$$
From all the above computations we obtain 
$$
\int_X (-\psi) \omega_{\varphi}\wedge\omega^{n-1} \leq (A + V  +3BA ) + 2 K_1 A^{1/2} \left(A(\psi)\right)^{1/2}.
$$
Taking the supremum over all such $\varphi$ we arrive at 
\[
A(\psi) \le (A + V  +3BA )) + 2 K_1A^{1/2} \left(A(\psi)\right)^{1/2},
\]
which implies that $A(\psi)$ is uniformly bounded from above, independently of $\psi$. Thus, the bound holds for $k=1$. Suppose now that for all $j \in \mathbb{N}^{+}$, with $1\le 
j \le k-1$, and for all $\psi \in \SH_m(X,\omega), \varphi \in \PSH(X,\omega)$ with normalizing condition as above, the following inequality holds:
\[
\int_{X} \left| \psi  \right| \omega_{\varphi}^j \wedge \omega^{n-j} \le C_j.
\]
We need to infer the following inequality
\[
\int_{X} \left| \psi  \right| \omega_{\varphi}^{k} \wedge \omega^{n-k} \le C_{k}.
\]
Since
\[
\int_{X} \left| \psi  \right| \omega_{\varphi}^{k} \wedge \omega^{n-k} \\
= \int_{X} \left| \psi  \right| \omega \wedge \omega_{\varphi}^{k-1} \wedge \omega^{n-k} + \int_{X} \left| \psi  \right| (d d^c \varphi) \wedge \omega_{\varphi}^{k-1} \wedge \omega^{n-k},
\]
using induction hypothesis it is enough to estimate the second term. The integration by parts (Proposition~\ref{For:IPP}) yields
\begin{equation} \label{equ:kipp}
\begin{aligned}
\int_{X} \left(-\psi  \right) (d d^c \varphi) \wedge \omega_{\varphi}^{k-1} \wedge \omega^{n-k} 
=& \int_{X} \left(- \varphi  \right) d d^c (\psi \wedge \omega_{\varphi}^{k-1} \wedge \omega^{n-k} ) \\
=& \int_{X} \left(- \varphi  \right) d d^c \psi \wedge \omega_{\varphi}^{k-1} \wedge \omega^{n-k} + 2 \int_{X} \left(- \varphi  \right) d \psi \wedge d^c\left( \omega_{\varphi}^{k-1} \wedge \omega^{n-k} \right) \\
&+ \int_{X} \left(- \varphi  \right) \psi  d d^c \left( \omega_{\varphi}^{k-1} \wedge \omega^{n-k} \right) .
\end{aligned}
\end{equation}
We now deal with the first term of \eqref{equ:kipp}. Observe at first that
\[
\begin{aligned}
\int_{X} \left(- \varphi  \right) d d^c \psi \wedge \omega_{\varphi}^{k-1} \wedge \omega^{n-k} 
=& \int_{X} \left(- \varphi  \right) \omega_{\psi} \wedge \omega_{\varphi}^{k-1} \wedge \omega^{n-k} - \int_X \left(- \varphi  \right) \omega \wedge \omega_{\varphi}^{k-1} \wedge \omega^{n-k}\\
\le& \int_{X} \left(- \varphi  \right) \omega_\psi \wedge \omega_{\varphi}^{k-1} \wedge \omega^{n-k}\\
\leq & \int_X \omega_\psi \wedge \omega_{\varphi}^{k-1} \wedge \omega^{n-k}\\
=& \int_X \omega \wedge \omega_{\varphi}^{k-1} \wedge \omega^{n-k} + \int_X d d^c \psi \wedge (\omega_{\varphi}^{k-1} \wedge \omega^{n-k})  ,
\end{aligned}
\]
where the first inequality follows from the assumption that $\varphi \le 0$, and that $\omega_{\varphi}^{k-1} \wedge \omega^{n-k+1}$ is positive since $\varphi \in \PSH(X,\omega)$ while the second inequality holds because of the normalization of $\varphi$ and the positivity of $ \omega_{\varphi}^{k-1}\wedge \omega^{n-k}$. 
Note that
\[
 \int_X \omega \wedge \omega_{\varphi}^{k-1} \wedge \omega^{n-k} = \sum_{j=0}^{k-1} a_{k,j}\int_X \omega^{j+1 } \wedge (\ddc \varphi)^{k-1-j } \wedge \omega^{n-k} .
\]
For each term in the sum, integration by parts, along with the bound on $\varphi$, leads to
\[
\int_X \omega \wedge \omega_{\varphi}^{k-1} \wedge \omega^{n-k} \le K_2.
\]
Moreover, an elementary computation yields:
\begin{equation} \label{formula:elecom}
\begin{aligned}
\ddc \left(\omega_{\varphi}^{k-1} \wedge \omega^{n-k}\right)
=& (k-1)(k-2) \omega_{\varphi}^{k-3} \wedge d \omega \wedge d^c \omega \wedge \omega^{n-k} \\
&+ (k-1) \omega_{\varphi}^{k-2} \ddc \omega \wedge \omega^{n-k}
+ 2 (k-1) \omega_{\varphi}^{k-2} \wedge d \omega \wedge d^c \omega^{n-k}\\
&+ \omega_{\varphi}^{k-1} \wedge \ddc \omega^{n-k} .
\end{aligned}
\end{equation}
Thus, performing integration by parts, it follows from \eqref{for:B} and the induction hypothesis that
\[
\begin{aligned}
\int_{X}  d d^c \psi \wedge  \left(\omega_{\varphi}^{k-1} \wedge \omega^{n-k} \right)
= &\int_X \psi d d^c (\omega_{\varphi}^{k-1} \wedge \omega^{n-k}) \\
\le& K_3 \int_X (-\psi)(\omega_{\varphi}^{k-1} \wedge \omega^{n-k+1} + \omega_{\varphi}^{k-2} \wedge \omega^{n-k+2} + \omega_{\varphi}^{k-3} \wedge \omega^{n-k+3})\\
\le& K_3 (C_{k-1}+C_{k-2}+ C_{k-3}).
\end{aligned}
\]
We can also infer from \eqref{formula:elecom} and the induction hypothesis that the third term of \eqref{equ:kipp} is uniformly bounded:
\begin{equation} \label{formula:3rdt}
\begin{aligned}
\int_X (-\varphi)\psi \ddc \left( \omega_{\varphi}^{k-1} \wedge \omega^{n-k}\right) 
\le& K_4 \int_X(-\varphi) \psi(\omega_{\varphi}^{k-1} \wedge \omega^{n-k+1} + \omega_{\varphi}^{k-2} \wedge \omega^{n-k+2} + \omega_{\varphi}^{k-3} \wedge \omega^{n-k+3})\\
\le& K_4 \sup_{X}|\varphi| (C_{k-1}+ C_{k-2}+ C_{k-3}) \\
\le& K_5.
\end{aligned}
\end{equation}
We now control the second term of \eqref{equ:kipp}. Stokes' theorem yields
\[
\begin{aligned}
\int_{X} \left(- \varphi  \right) d \psi \wedge d^c \left( \omega_{\varphi}^{k-1} \wedge \omega^{n-k} \right)  
= \int_{X} \psi   d \varphi \wedge d^c\left( \omega_{\varphi}^{k-1} \wedge \omega^{n-k} \right)\\ + \int_{X} \psi  \varphi d d^c\left( \omega_{\varphi}^{k-1} \wedge \omega^{n-k} \right),
\end{aligned}
\]
where the second term is bounded by $K_5$, for the same reason as in \eqref{formula:3rdt}.
Using the Cauchy--Schwarz inequality (Lemma~\ref{lem:cs}) for the first term, we obtain that
\begin{equation}
\begin{aligned}
& -\int_{X} \left(- \psi  \right)   d \varphi \wedge d^c\left( \omega_{\varphi}^{k-1} \wedge \omega^{n-k} \right) \\
=&- \int_X \left(- \psi  \right)   d \varphi \wedge \left( (k-1) \omega_{\varphi}^{k-2} \wedge d^c \omega \wedge \omega^{n-k} + (n-k) \omega_{\varphi}^{k-1} \wedge d^c \omega \wedge \omega^{n-k-1} \right) \\
\le &  \left( K_6 \int_X (- \psi) d\varphi \wedge d^c \varphi \wedge \omega \wedge T \right)^{1/2} \left( \int_X (- \psi) \omega^2 \wedge T \right)^{1/2}, \label{1}
\end{aligned}    
\end{equation}
where $T= (k-1) \omega_{\varphi}^{k-2} \wedge \omega^{n-k} + (n-k) \omega_{\varphi}^{k-1} \wedge \omega^{n-k-1}$ is positive. Applying the induction hypothesis, $\int_X (- \psi) \omega^2 \wedge T$ can be controlled by $ ((k-1) C_{k-2}+ (n-k) C_{k-1})$.

By the induction hypothesis, and applying \eqref{dweddc} to \eqref{1}, we arrive at:
\[
\begin{aligned}
&\int_{X} \left(- \psi  \right)   d \varphi \wedge d^c\left( \omega_{\varphi}^{k-1} \wedge \omega^{n-k} \right) \\
\le& K_6^{1/2} \left((k-1) C_{k-1}+ (n-k) C_{k-1} \right)^{1/2} \left( \int_X (- \psi) \left( \omega + \frac{1}{2} d d^c (\varphi +1)^2 \right)\wedge \omega \wedge T \right)^{1/2} \\    
\le& K_7 \left( \int_X (- \psi) \left( \omega + \frac{1}{2} d d^c (\varphi +1)^2 \right)\wedge \omega \wedge T \right)^{1/2} .
\end{aligned}
\]
To simplify the notation, we write $T_1 = (k-1)\omega_{\varphi}^{k-2} \wedge \omega^{n-k+1} $, $T_2= (n-k) \omega_{\varphi}^{k-1} \wedge \omega^{n-k}$, and
\[
\begin{aligned}
A_1 \coloneqq   \int_X (- \psi ) \omega_u \wedge T_1 , A_2 \coloneqq  \int_X (- \psi ) \omega_u \wedge T_2 , u= \frac{(\varphi+1)^2}{2}-1 . 
\end{aligned}
\]
Set $v_1 \coloneqq \frac{u+(k-2)\varphi}{k-1}$, and $v_2 \coloneqq \frac{u +(k-1) \varphi}{k}$. We remark that $-1\le v_1, v_2 \le 0$, and that
\[
\begin{aligned}
\omega_u , \omega_{\varphi } \le& (k-1) \omega + \ddc \left(  u +(k-2)\varphi \right);\\
\omega_u , \omega_{\varphi } \le& k \omega + \ddc \left(  u +(k-1)\varphi \right).
\end{aligned}
\]
By induction hypothesis, we obtain that
\[
\begin{aligned}
 A_1 \le& (k-1)^{k-1} (k-1) \int_X (-\psi)\omega_{v_1}^{k-1} \wedge \omega^{n-k+1} \le (k-1)^{k-1} (k-1)C_{k-1};\\
 A_2 \le& k^{k} (n-k) \int_X (-\psi) \omega_{v_2}^{k} \wedge \omega^{n-k} .   
\end{aligned}
\]
Set $S(\psi)= \sup \left\{ \int_X (-\psi)\omega_v^{k} \wedge \omega^{n-k} : -1\le v \le 0, v\in \PSH(X,\omega)\right\}$. From the preceding computations, we conclude that
\[
\int_X (-\psi) \omega_{\varphi}^k \wedge \omega^{n-k} \le K_8 + \left( K_9 S(\psi)\right)^{1/2},
\]
where $K_8, K_9$ are independent on $\psi$.
Taking the supremum over all $-1 \le \varphi \le 0$, we obtain 
\[
S(\psi) \le K_8 + \left( K_9 S(\psi)\right)^{1/2},
\]
which leads to an uniform upper bound for $\int_X (-\psi)\omega_v^{k} \wedge \omega^{n-k}$ for $-1 \le v \le 0$. The proof is thereby complete.
\end{proof}

\begin{coro} \label{3.2}
There exists a constant $C > 0$ such that for all $\psi \in \SH_m(X,\omega)$ satisfying $\sup_X \psi = -1 $ and for every $t > 0$ we have
\[
\widetilde{\Capa}_{\omega,m}(\psi < -t) \le C/t.
\]
 \end{coro}

\begin{proof}
We fix $\varphi\in \PSH(X,\omega)$ satisfying $-1 \le \varphi \le 0$.

By lemma~\ref{3.1}, we deduce that
\[
\int_{(\psi < -t)} (-\psi) \omega^{m}_{\varphi} \wedge \omega^{n-m} \le C.
\]
We also observe that
\[
\int_{(\psi < -t)} (-\psi) \omega^{m}_{\varphi} \wedge \omega^{n-m}\ge \int_{(\psi < -t)} t \omega^{m}_{\varphi}\wedge \omega^{n-m},
\]
and this completes the proof.
\end{proof}

\begin{proof}[Proof of the main theorem]
    Without loss of generality, we may assume that $\sup_X \varphi =-1$. Fix $p < \frac{n}{n-m}$ and $q$ such that $p <q <\frac{n}{n-m}$. A fundamental calculation yields that 
    \[
    \begin{aligned}
        \int_X(-\varphi)^p \omega^n =&  \int_{0}^{+\infty}  \Vol(\{ (-\varphi)^p >t \})dt\\
        =& \int_{0}^{1}  \Vol(X)dt+\int_{1}^{+\infty}  \Vol(\{ (-\varphi)^p >t \})dt \\
        =& \int_X \omega^n +\int_{1}^{+\infty} p \Vol(\{ (-\varphi) >s \}) s^{p-1}ds .
    \end{aligned}
    \]
    It follows from the volume-capacity estimate (Proposition~\ref{prop:V-Cest}) and Corollary~\ref{3.2} that 
    \[
    \begin{aligned}
    \int_{1}^{+\infty} p \Vol(\{ (-\varphi) >s \}) s^{p-1}ds
        \le& C_q p\int_{1}^{+\infty} s ^{p-1} \left(\widetilde{\Capa}_{\omega,m}(\{ \varphi <-s \})\right)^qds \\
        \le & C_{p,q}\int_{1}^{+\infty} s^{p-1-q}ds < +\infty,
    \end{aligned}
    \]
where $C_{p,q}$ depend only on $p,q$.
\end{proof}

\bibliographystyle{smfalpha_new}
\bibliography{ref}

\end{document}